\documentclass[12pt,reqno]{amsart}
\usepackage{amscd,amsmath,amssymb,amsthm,url,cleveref}
\usepackage[utf8]{inputenc}
\usepackage[russian,english]{babel}

\title[On the classification of matrix distributions]{%
On the classification problem of matrix distributions of measurable functions in several variables
}
\author{Anatoli M. Vershik}
\author{Ulrich Hab\"ock}
\address{Anatoli M. Vershik, Steklov Inst.of Mathematics, St.Petersburg, Fontanka 27, St.Petersburg, 191023; Math.Dept of St.Petersburg State University;  Russia}\email{vershik@pdmi.ras.ru}
\address{Ulrich Hab\"ock,  Cometence Centre for IT--Security, Fachhochschule Campus Wien, Favoritenstrasse 226, A--1100 Wien,  Austria}\email{ulrich.haboeck@fh-campuswien.ac.at}

\date{\today}
\thanks{The first author supported by the Russian Science Foundation grant \# 14-11-00581.}

\newtheorem{thm}{Theorem}[section]
\newtheorem*{thm*}{Theorem}
\newtheorem{cor}[thm]{Corollary}
\newtheorem{lem}[thm]{Lemma}
\newtheorem{prop}[thm]{Proposition}

\theoremstyle{definition}
\newtheorem{defn}[thm]{Definition}
\newtheorem*{defn*}{Definition}

\theoremstyle{remark}
\newtheorem{rem}[thm]{Remark}

\newtheorem{rem*}[]{Remark}

\newcommand{\comment}[1]{}

\newcommand{\N}{\mathbb N}
\newcommand{\Z}{\mathbb Z}

\newcommand{\R}{\mathbb R}
\newcommand{\mc}{\mathcal}
\newcommand{\mf}{\mathfrak}

\renewcommand{\S}{S}

\DeclareMathOperator{\Aut}{Aut}

\DeclareMathOperator{\id}{id}
\DeclareMathOperator{\meas}{Meas}

\begin{document}
\setlength{\parskip}{2mm}
\maketitle

\begin{center}
To the memory of Michael I. Gordin
\end{center}

\begin{abstract}
We resume the results from \cite{Vershik FA} on the classification of measurable functions in several variables, with some minor corrections of purely technical nature, and give a partial solution to the characterization problem of so--called matrix distributions, which are the metric invariants of measurable functions introduced in \cite{Vershik FA}.
The characterization of these invariants, considered as $\S_\N\times\S_\N$--invariant, ergodic measures on the space of matrices is closely related to Aldous' and Hoover's representation of row-- and column--exchangable distributions \cite{Aldous1981,Hoover1982}, but not in such an obvious way as was initially expected in \cite{Vershik FA}.
\end{abstract}

Key words: Matrix distribution, pure and completely pure function,
invariant measures.

\section{Introduction and outline of the paper}

The classification problem of measurable functions  is the question whether a measurable function
\[
f: X_1\times X_2 \times \ldots\times X_n \longrightarrow Z,
\]
defined in several variables from standard probability spaces
$(X_i,\mf B_i,\mu_i)$ and values in a Borel space $Z$,
 is isomorphic to another such function $h$ with arguments taken from other probability spaces $(Y_i,\mf C_i, \nu_i) $.
The notion of isomorphism refers to the category of measure spaces:
the existence of measure--preserving,  invertible Borel maps
\[
T_i: \left(X_i,\mf B_i,\mu_i\right) \longrightarrow \left(Y_i,\mf C_i, \nu_i\right)
\]
which carry the function $f$ to $h$ by seperate coordinate--wise application.
In terms of commutative diagrams:
the diagram
\[
\begin{CD}
X_1\times X_2\times\ldots\times X_n @>f>> \R\\
@V\text{meas.--pres.}V T= (T_1,T_2,\ldots, T_n)V @VV\text{id}V\\
Y_1\times Y_2 \times\ldots\times Y_n@>h>>\R
\end{CD}
\]
commutes on a set of full measure.
The classical case of functions in one argument was solved by Rokhlin \cite{Rokhlin57}, and is nowadays found in many modern textbooks on measure theory such as \cite{Bogachev07}:
roughly speaking, two functions are isomorphic if and only if the distribution as well as the
multiplicities of the attained values, described by the metric types of the conditional measures $\mu_z(\,\cdot\,)= \mu(\,\cdot\,|f=z)$, coincide (we will give the precise statement in Section~\ref{s:BasicDefinitions}).

When considering the isomorphism problem for functions in several arguments one obviously needs an entirely different concept.
This problem was posed in full generality by the first author in \cite{Vershik FA},  and a first application of the idea of matrix distributions was in the context of classifying metric triples, i.e. Polish spaces with fully supported probability measure, initiated by M. Gromov, cf. \cite{Gromov} and \cite{RandomMetricSpaces}.
The \textit{tensor distribution} $D_f$ of a measurable function $f$ (or \textit{matrix  distribution}  in the case of  two variables only) introduced in \cite{Vershik FA} is a probability measure
on the space of infinite tensors, i.e.
\[
D_f \in \meas_1\left( Z^{\N^n}\right),
\]
which arises as the distribution of the tensors
\[
\left(r_{i_1,i_2,\ldots,i_n}\right)= \left(f\left(x_1^{(i_1)}, x_2^{(i_2)},\ldots, x_n^{(i_n)}\right)\right)_{i_1,i_2,\ldots,i_n=1}^\infty
\]
determined by the $f$--values when the arguments are sampled independently and at random according to the given measures $\mu_i$.
This measure is invariant and ergodic with respect to action of the product
\[
\left(\S_\N\right)^n = \S_\N\times\S_\N\times\ldots\times\S_\N,
\]
of the infinite symmetric group $S_\N$,
acting independently on the indices of the tensors.
It is shown in \cite{Vershik FA} that $D_f$ is a complete metric invariant for the isomorphism problem of measurable functions,  provided
that the functions under consideration are \textit{pure}, which means that they do not admit non--trivial factors in the category of measurable functions (see Section~\ref{s:BasicDefinitions} for a precise statement of that property).


The characterization problem of those invariant measures
\[
\lambda\in \meas_1\left( \R^{\N\times\N}\right)
\]
which are the matrix distribution of a function  $f(x,y)$ in two variables is closely related to
Aldous' and Hoover's representation of exchangeable distributions on infinite arrays in two dimensions \cite{Aldous1981,Hoover1982}:
\textit{any array of random variables $(X_{i,j})$, $1\leq i,j<\infty$, with an $\left(\S_\N\times\S_\N\right)$--invariant joint distribution can be represented as function
\[
X_{i,j} = f\left(\alpha, \zeta_i,\eta_j,\xi_{i,j} \right)
\]
of underlying i.i.d. random variables $\alpha$, $\zeta_i$, $\eta_j$, and $\xi_{i,j}$.
}
The connection of Aldous' theorem with \cite{Vershik FA} consists in the fact that some $(S_\N\times\S_\N)$--invariant ergodic measures correspond to the matrix distribution of a measurable function, which is up to isomorphy unique -- a property that doesn't follow from the approach in \cite{Aldous1981, Hoover1982}.
Recently, the first author proved the same answer for arbitrary invariant ergodic measures which covers Aldous' representation in full generality \cite{Vershik15}, but we shall not touch this topic here.

In the present work we concetrate on the above mentioned partial solution to the characterization problem of matrix distributions, correcting the corresponding statement Theorem 3 in \cite{Vershik FA}:
an $\left(\S_N\right)^n$--invariant ergodic measure $\lambda$ is the matrix distribution of a \textit{completely pure} function, that is a  pure function with trivial \textit{congruence group}
\begin{multline*}
K_f= \left\{ \left(T_i\right)_{i=1}^n \in \prod_{i=1}^n \Aut_0(X_i,\mu_i): \right.
\\
f\left(T_1(x_1) ,\ldots, T_n(x_n)\right)=f(x_1,\ldots, x_n) \quad \text{a.e.} \Bigg\},
\end{multline*}
if and only if it is a \textit{simple measure}, i.e.  the ergodic components of the seperate actions of $\S_\N$ generate the entire sigma algebra in the space of tensors.

The paper is organized as follows:
In Section \ref{s:BasicDefinitions} we recall important definitions and facts from \cite{Vershik FA} and restate Rokhlins classification theorem for univalent functions.
Section \ref{s:PureFunctions} revises basic properties of pure functions from \cite{Vershik FA} including self--contained proofs.
We chose to repeat these elementary facts, as they are needed in Section \ref{s:CompleteInvariant}, in which we present a slightly modified proof of the completeness theorem Theorem 2 from \cite{Vershik FA}.
Finally, Section \ref{s:SimpleMeasures} elaborates the above mentioned partial characterization of matrix distributions via the so--called \textit{general canonical model} for a measurable function.
As the  case of functions in more than two variables bears no additional obstacles from the conceptual point of view, we shall restrict ourselves throughout Sections \ref{s:BasicDefinitions}--\ref{s:SimpleMeasures} to the case of two variables
\[
n=2
\]
only.
The general case, which is then obtained by a straight--forward generalization of our methods, is briefly discussed in Section \ref{s:more than two}.

The present paper is a revised version of a chapter taken form the thesis \cite{HaboeckDiss}, which originated from a discussion on the classification of matrix distributions during the course \textit{Measure theoretic constructions and their applications in ergodic theory, asymptotics, combinatorics, and geometry} given by the first author  in autumn 2002 at the Erwin Schr\"odinger Institute, Vienna.

\subsection*{Acknowledgements}
The second author would like to thank the first author for his endless patience in discussions and correspondence, in particular during his last stage of writing his thesis.


\section{Basic definitions and facts}
\label{s:BasicDefinitions}

Throughout the following we consider all spaces to be standard probability spaces, i.e. standard Borel spaces $(X,\mf B)$ equipped with a Borel probability measure $\mu$.
To avoid cumbersome notation we shall write $(X,\mu)$ (or  just $X$ if it is clear to what measure on $X$ we refer) instead of $(X,\mf B,\mu)$, whenever it is convient.
All functions are considered to be measurable unless the contrary is explicitely stated.

\subsection{Isomorphy, factors, and pure functions}
We call two measurable functions $f:(X,\mu)\times (Y,\nu) \longrightarrow\R$ and $f':(X',\mu')\times (Y',\nu') \longrightarrow\R$ to be \textit{isomorphic} if we can find measure preserving isomorphisms $S:X\longrightarrow X'$ and $T:Y\longrightarrow Y'$ such that
\[
f'\big(S(x),T(y)\big)= f(x,y) \quad\text{a.e.},
\]
where `a.e.' refers to the product measure $\mu\times\nu$.
Whenever the transformations $S$ and $T$ are measure preserving projections (and not necessarily invertible $\bmod\, 0$), i.e. they map onto a set of full measure, we say that $f'$ is a \textit{factor} of the function $f$.

Notice that isomorphy as well as being a factor is a notion on the equivalence classes ($\bmod\, 0$) of functions.
In terms of commutative diagrams, if $f'$ is a factor of $f$, then the diagram
\[
\begin{CD}
X @.\times @. Y 	@>f>>	 Z
\\
@V\text{meas.pres. projections}VSV@.@VVTV 		@VV\text{id}V\\
X' @.\times @. Y' @>f'>>Z
\end{CD}
\]
commutes.

\begin{defn}
\label{d:pure function}
A measurable function $f:(X,\mu)\times (Y,\nu)\longrightarrow Z$ is \textit{pure} if it admits no true factor, by which we mean that every factor $f':(X',\mu')\times (Y',\nu')\longrightarrow Z$ of $f$ is already isomorphic to $f$.
\end{defn}

We denote by $\mf B(X,Z)$ (or $\mf B(Y,Z)$) the space of all equivalence classes $\bmod\, 0$ of measurable functions from $X$ (or $Y$, respectively) into the standard Borel space $Z$, endowed with topology of convergence in measure with respect to any Polish topology generating the Borel structure of $Z$.
Since $f$ is measurable so are the mappings
\begin{align*}
f_X: X\longrightarrow \mf B(Y,Z), \quad x &\mapsto \left[f(x,\,\cdot\,)\right]_\nu
\\\intertext{and}
f_Y: Y\longrightarrow \mf B(X,Z), \quad y &\mapsto \left[f(\,\cdot\,, y)\right]_\mu,
\end{align*}
where $[\:.\:]_\nu$ and  $[\:.\:]_\mu$ denote the corresponding equivalence class.
For brevity, we will omit the brackets in the sequel.

It is evident from Definition \ref{d:pure function} that pureness of a function can be rephrased as follows.
\begin{lem}\label{l:pure}
A function $f:X\times Y\longrightarrow Z$ is pure if and only if both mappings
$f_X: X\longrightarrow \mf B(Y,Z)$, $x\mapsto f(x,\,\cdot\,)$, and $f_Y: Y\longrightarrow \mf B(X,Z)$, $y\mapsto f(\,\cdot\,,y)$ are one--to--one on a set of full measure.
\end{lem}

\subsection{Rokhlins Theorem}
It is evident that two univalent measurable functions $f:(X,\mu)\longrightarrow Z$ and $f':(X',\mu')\longrightarrow Z$, i.e. one--to--one on a set of full measure, are isomorphic if and only if their distributions $D_{f}= \mu\circ f^{-1}$ and $D_{f'}= \mu'\circ f'^{-1}$ coincide.
If the functions under consideration are not univalent one has to take in account the `multiplicity' certain values are obtained.
This is done by looking at the function
\[
m_f: Z\longrightarrow \Sigma =\Big\{(c_i)_{i\geq 1}: 0\leq c_i \leq 1, \sum_{i=1}^\infty c_i\leq 1 \Big\}
\]
which maps any value $z$ to the \emph{metric type} $m_f(z)$ of the conditional probability distribution
\[
\mu_z(\,\cdot\,)= \mu(\,\cdot\,|f=z),
\]
which is the sequence of weights $\big\{c_i=\mu_z(\{a_i\})\big\}_{i\geq 1}$ of the atoms $a_i$ of the measure $\mu_z$ arranged in a non-increasing way.
Note that since the conditional probability distributions are defined uniquely (mod $0$) so is the function $m_f$.
\begin{thm}[Rokhlin,\cite{Rokhlin57}]\label{t:Rokhlin}
Assume that $f_i:X_i\longrightarrow Z$ ($i=1,2$) are two measurable functions defined on standard probability spaces $(X_i,\mu_i)$. Then there exist an isomorphism $T:X_1\longrightarrow X_2$ of the measure spaces with the property that $\mu_1\circ T^{-1} = \mu_2$ and $f_2\circ T(x)= f_1(x)$ almost everywhere if and only if their extended functions
\begin{equation*}
f_i^{e}: X_i\longrightarrow Z\times\Sigma, \quad x\mapsto \big(f_i(x),m_{f_i}\circ f_i(x)\big)
\end{equation*}
have the same distribution, by which we mean that the measures $D_1=\mu_1\circ (f_1^e)^{-1}$ and $D_2=\mu_2\circ (f_2^e)^{-1}$ coincide.
\end{thm}

\subsection{Group actions and ergodic decompositions}
\label{s:GroupActionsErgodicDecomposition}
The product
$
 \S_\N\times \S_\N
$
of the infinite symmetric group
\[
\S_\N = \bigcup_{n=1}^\infty \S_{\{1,\ldots,n\}}
\]
of all finite permutations of $\N$ acts on the product space
\[
(X\times Y)^\N = X^\N\times Y^\N
\]
in the canonical way by acting independently on the indices of the sequences, i.e.
\[
g \cdot \left(\left(x_i\right)_{i=1}^\infty, \left(y_j\right)_{j=1}^\infty\right) = \left( \left(x_{g_1^{-1}(i)}\right)_{i=1}^\infty, \left(y_{g_2^{-1}(j)}\right)_{j=1}^\infty \right)
\]
for every $\left((x_i),(y_i)\right)$ from $X^\N\times Y^\N$ and $g=(g_1,g_2)$ from  $S_\N\times S_\N$.
Its action on the space of infinite matrices is that of  permuting  rows and columns separately:
\[
g\cdot \left(r_{i,j}\right) = \left(r_{g_1^{-1}(i), g_2^{-1}(j)}\right)_{i,j}
\]
for every  $g=(g_1,g_2)$ from  $S_\N\times S_\N$ and matrix $(r_{i,j})$.

Let $G$ be a countable semigroup acting measurably on a standard Borel space, which will be in our case the space of matrices $Z^{\N\times\N}$.
Every $G$--invariant probability measure $D$ on $Z^{\N\times\N}$ is then decomposed into $G$--invariant ergodic measures $D_r$ by a (mod 0) uniquely defined Borel mapping
\[
\pi: Z^{\N\times\N} \longrightarrow\meas_1\left(Z^{\N\times\N},G\right), \quad r\mapsto D_r^{G},
\]
which maps into the standard Borel space of $G$--invariant probability measures, endowed with the topology of weak convergence, such that the formula
\[
D(B) = \int_{Z^{\N\times\N}} D_r^{G} (B) \cdot dD(r)
\]
holds for every Borel set $B\subseteq Z^{\N\times\N}$.

\subsection{De Finetti's theorem}
We shall make use of de Finett's theorem on exchangeable distributions in the following form:
\textit{
Every Borel probability measure $m$ on the product $X^\N$ of a Borel space $X$, which is invariant and ergodic with respect to the action $\S_\N$ is Bernoulli, i.e.
\[
m = \mu^\N
\]
for a Borel probability measure $\mu$ on $X$.
}
There is a very simple and elegant proof of de Finetti's theorem with help of the point--wise ergodic theorem with respect to the countable group $\S_\N$, but we shall not dwell on this,--- \cite{De Finetti}.

\subsection{Matrix distribution of a measurable function}
Let us recall the definition of matrix distributions from \cite{Vershik FA}:
\begin{defn}
Let $f:(X,\mu)\times (Y,\nu) \longrightarrow Z$ be a measurable function in two variables.
Its \textit{matrix distribution} $D_f$ is the pushforward measure
\begin{equation*}
D_f = (\mu^\N\times \nu^\N)\circ F_f^{-1}
\end{equation*}
of the Bernoulli measure $\mu^\N\times \nu^\N$ under the evaluation function
\begin{equation*}
F_f : X^\N \times Y^\N \longrightarrow Z^{\N\times\N},
\quad\big((x_i), (y_j)\big) \mapsto \big(f(x_i,y_j)\big)_{i,j}.
\end{equation*}
\end{defn}
This definition generalizes the notion of matrix distributions for Gromov triples, i.e. Polish spaces with probability measure introduced in \cite{Vershik metric triples}.
Note that $D_f$ is a probability measure on the space of infinite matrices $Z^{\N\times \N}$, and that
$F_f$ is equivariant with respect to the action of $\S_\N\times\S_\N$, i.e.
\[
F_f \big(g \cdot\big(\left(x_i\right),\left(y_j\right)\big)\big) = g \cdot F_f \big(\left(x_i\right), \left(y_j\right)\big).
\]
Hence invariance and ergodicity of the Bernoulli measure $\mu^\N\times\nu^\N$ yields invariance and ergodicity of the matrix distribution.

\subsection{The congruence group of a measurable function}

The \textit{congruence group} of a function in two arguments is the group of measure--preserving symmetries
\begin{multline*}
K_f= \Big\{ (S,T) \in \Aut_0\left(X,\mu\right)\times\Aut_0\left(Y,\nu\right):
\\
f\big(S(x) ,T(y)\big)=f(x,y) \quad \text{a.e.} \Big\}.
\end{multline*}
This group plays an important role for our partial solution to the characterization problem of matrix distributions, Theorem \ref{t:simple measures}.
It is remarkable that the congruence group  is compact, when endowed with the weak topology.
This fact is shown in \cite{Congruence group}, but we will not make use of it in the sequel.

As in \cite{Vershik13} we shall call any pure function $f$ with trivial congruence group simply \textit{completely pure function}.


\section{Reduction of the classification problem to that of pure functions}
\label{s:PureFunctions}

Any measurable function $f$ has a pure factor $\bar f$ defined in a natural way:
As the functions $f_X$ and $f_Y$ are Borel the equivalence relations defined by
\begin{equation*}
\zeta_X= \big\{\left(x_1,x_2\right)\in X\times X :  f(x_1,\,\cdot\,) = f(x_2,\,\cdot\,) \: (\bmod\,\nu)\big\}
\end{equation*}
and
\begin{equation*}
\zeta_Y = \big\{\left(y_1,y_2\right)\in Y\times Y : f(\,\cdot\,,y_1) = f(\,\cdot\,,y_2) \: (\bmod\, \mu)\big\}
\end{equation*}
are partitions of the respective spaces into measurable components.
In order to stay within the category of standard measure spaces we define\footnote{It is no good choice to define the factor $X/\zeta_X$ to be the set of all equivalence
classes with respect to $\zeta_X$ and its sigma algebra the algebra of all $\zeta_X$-saturated $\mf S$-measurable sets; this space need not to be standard Borel.} the factor spaces
\[
X/\zeta_X \quad\text{and}\quad Y/\zeta_Y
\]
as the standard Borel space $\mf B(Y,Z)$ and $\mf B(X,Z)$ with the factor mappings
\begin{align*}
\pi_X: X\longrightarrow X/\zeta_X, \quad &\pi_X=f_X,
\intertext{and}
\pi_Y: Y\longrightarrow Y/\zeta_Y, \quad & \pi_Y=f_Y,  
\end{align*}
and the projected measures $\mu\circ \pi_X^{-1}$ and $\nu\circ\pi_Y^{-1}$ as measures, respectively. 
The mapping $f_X$ as function from $X$ into the space $\mf B(Y,Z)$ translates under the projection $\pi_X$ to the identity mapping regarded as function from $X/\zeta_X$ to $\mf B(Y,Z)$, and by the measurability it originates from a jointly measurable function
\[
f':X/\zeta_X \times Y\longrightarrow Z
\]
which by construction satisfies that
\[
f'\big(\pi_X(x), y\big) = f(x,y) \quad\text{a.e.},
\]
and the corresponding map $f'_{X/\zeta_X}$ is one--to--one on a set of full measure.
In the same manner we proceed with the second argument and finally end with a measurable function
\begin{equation*}
\bar f: X/\zeta_X \times Y/\zeta_Y\longrightarrow Z
\end{equation*}
such that
\[
\bar f\big(\pi_X(x),\pi_Y(y)\big) = f(x,y) \quad\text{a.e.}.
\]
This shows that the diagram
\[
\begin{CD}
X\times  Y @>f>> Z
\\
@V\text{meas.--pres.}V\pi= (\pi_X,\pi_Y)V @VV\text{id}V
\\
X/\zeta_X\times Y/\zeta_Y@>f'>>Z
\end{CD}
\]
is commutative on a set of full measure, and since the corresponding mappings $\bar f_{X/\zeta_X}$ and $\bar f_{Y/\zeta_Y}$ are now one--to--one on a set of full measure, the function $\bar f$ is consequently pure.
It is not difficult to see that this pure factor $\bar f$ is uniquely determined up to isomorphy. We will sometimes refer to it as the \textit{purification} or the \textit{unique pure factor} of $f$.

The reduction of the isomorphism problem is done in the spirit of Rokhlin's Theorem \ref{t:Rokhlin}.
Instead of $f$ we consider its extended pure factor
\begin{equation*}
\bar f^e: X/\zeta_X \times Y/\zeta_Y \longrightarrow Z\times \Delta^2, \quad (x,y)\mapsto \left(\bar f(x , y), m(\mu_{x}), m(\nu_{y})\right),
\end{equation*}
where $m(\mu_{x})$ and $m(\nu_{y})$ denote the metric types of the respective ($\bmod\, 0$ uniquely determined) conditional measures $\mu(\,\cdot\,|\pi_X = x)$ and $\nu(\,\cdot\,|\pi_Y = y)$.

\begin{thm}
\label{t:ReductionToPureCase}
Two not necessarily pure functions $f:(X,\mu)\times (Y,\nu)\longrightarrow Z$ and $g: (X',\nu')\times (Y',\nu') \longrightarrow Z$ are isomorphic if and only if their extended pure factors
\begin{equation*}
\bar f^e: X/\zeta_X \times Y/\zeta_Y \longrightarrow Z\times \Delta^2, \quad (x,y)\mapsto \big(\bar f(x , y), m(\mu_{x}), m(\nu_{y})\big),
\end{equation*}
and
\begin{equation*}
\bar g^e: X'/\zeta_{X'} \times Y'/\zeta_{Y'} \longrightarrow Z\times \Delta^2, \quad (x', y')\mapsto \big(\bar g(x', y'), m(\mu'_{x'}), m(\nu'_{y'})\big),
\end{equation*}
both defined as above, are isomorphic.
\end{thm}
\begin{proof}
The statement of the theorem is clear since every measure preserving isomorphism $\bar T:X/\zeta_X\longrightarrow X'/\zeta_X'$ which carries the function $x\mapsto m(\mu_{x})$ to the function $x'\mapsto m(\mu'_{x'})$ ($\bmod\, 0$) can be lifted to a measure preserving isomorphism $T:X\longrightarrow X'$ so that $T= \bar T\circ\pi_X$ ($\bmod\, 0$) and the same is true for the second coordinate spaces $Y$ and $Y'$.
\end{proof}

\section{The matrix distribution as complete invariant for pure functions}\label{s:complete invariant}
\label{s:CompleteInvariant}

In this section we resume the individual canoncial model and the completeness theorem for measurable functions in two variables.
Are results and proofs, apart from some technical details which are elaborated in greater detail,  can be  also found in \cite{Vershik FA}.

Let us start with an auxiliary lemma on pure functions.
We say that a sigma algebra $\mf S\subseteq\mf B_X$ equals modulo a measure $\mu$ (or simply `$\bmod\, 0$', whenever it is clear to which measure we refer) the whole Borel algebra $\mf B_X$, if its measure algebra $\mf S_\mu=\{[A]_\mu: A\in\mf S\}$, with $[A]_\mu=\{B\in\mf B_X: \mu(B\Delta A)=0\}$, coincides with the measure algebra defined by $\mf B_X$ itself.

\begin{lem}\label{l:sev:pure}
Suppose $f:X\times Y\rightarrow Z$ is a pure function. Then the following properties hold:
\begin{enumerate}
\item\label{i:pure} For any Borel set $Y'\subseteq Y$ which is of full $\nu$--measure the set of functions $\left\{f\left(\,\cdot\,, y\right): y\in Y'\right\}$ generates ($\bmod\, 0$)
the Borel algebra of $X$.
\item\label{i:pure sequencial} For $\nu^\N$--almost every sequence $\left(y_j\right)_{j=1}^\infty$ the countable collection of functions $\left\{f(\,\cdot\,,y_j): j\geq 1\right\}$ generates ($\bmod\, 0$) the Borel algebra of $X$.
\end{enumerate}
By symmetry the same statements hold when interchanging the role of $Y$ and $X$.
\end{lem}

\begin{proof}
As the restriction of $f$ to $X\times Y'$ is also a pure function, it is sufficient to prove (\ref{i:pure}) for the case $Y'=Y$.
Let $\mf F_Y$ denote the sigma algebra generated by the set of functions $\left\{f_y=f\left(\,\cdot\,,y\right) : y\in Y\right\}$.
Using standard arguments one sees that the function $f$ is measurable with respect to the sigma algebra $\mf F_Y\times \mf B_Y$, where $\mf B_Y$ is the Borel algebra of $Y$, after a  modification on a set of measure zero if necessary.
Thus the mapping
\[
f_X: X\longrightarrow \mf B(Y,Z), \quad x\mapsto f\left(x,\,\cdot\,\right),
\]
is  measurable ($\bmod\, 0$) with respect to the sigma algebra $\mf F_Y$\footnote{by which we mean that it coincides on a set of full measure with an $\mf F_Y$-measurable function}.
Hence injectivity of the map $f_X$ on a set of full measure implies that $\mf F_Y$ coincides modulo null sets with the entire Borel algebra of $X$.

To prove (\ref{i:pure sequencial}) let us choose a countable base $\left\{O_n\right\}_{n\geq 1}$ for the topology in $\mf B(Y,Z)$.
Now $\nu^\N$--almost every sequence $\left(y_j\right)$ is such that for any set $O_n$ of positive measure $\mu\circ f_X^{-1}$, its intersection with $\{f_{y_j}:j\geq 1\}$ is non--empty.
For every such sequence $\left(y_j\right)$ the sigma algebra $\mf F_{\left(y_j\right)}$ generated by the functions $\left\{f_{y_j}: j\geq 1\right\}$ contains ($\bmod\, 0$) the sigma algebra generated by the collection $\left\{f_y: y\in Y'\right\}$ with $Y'$ being the preimage of the support of the measure $\nu\circ f_Y^{-1}$ under the map $f_Y$.
As this set has full measure it follows from (\ref{i:pure}) that $\mf F_{\left(y_j\right)}$ coincides ($\bmod\, 0$) with the entire Borel algebra of $X$.
\end{proof}

\subsection{Individual canonical model of a measurable function}
In the sequel we assume that that $f:X\times Y\longrightarrow Z$ takes values in the interval $Z=[0,1]$.
This means no loss in generality, as any Borel space is measurably isomorphic to $[0,1]$ or to an -- at most -- countable subset.
It is an immediate consequence of Lemma \ref{l:sev:pure} that for $\mu^\N$--almost every sequence $(x_i)$ and $\nu^\N$--almost every sequence $(y_j)$ the mappings
\begin{align*}
L_{(y_j)} : X \longrightarrow [0,1]^\N,\quad& x\mapsto \big(f(x,y_j)\big)_{j},
\\
L_{(x_i)} : Y \longrightarrow [0,1]^\N,\quad& y\mapsto \big(f(x_i,y)\big)_{i}
\end{align*}
are one--to--one on a set of full measure, and therefore they are isomorphisms between the measure spaces $(X,\mu)$, $(Y,\nu)$ and the spaces
\begin{align*}
\left(X_{(y_j)},\mu_{(y_j)}\right) &= \left([0,1]^{\N},\:\mu\circ L_{(y_j)}^{-1} \right),
\\
\left(Y_{(x_i)},\nu_{(x_i)}\right) &= \left([0,1]^{\N},\:\nu\circ L_{(x_i)}^{-1} \right),
\end{align*}
respectively.
These spaces together with the function
\begin{equation*}
f_{(x_i),(y_j)} = f\circ\left (L_{(y_j)}\times L_{(x_i)}\right)^{-1}
\end{equation*}
form the \textit{canonical representation} (or \textit{canonical model}) of our measurable function $f$.
We shall call this model \textit{individual canonical model}, as both measures $\mu_{(y_j)}$, $\nu_{(x_j)}$ and the function $f_{(x_i),(y_j)}$, being the density of the absolutely continuous measure
\begin{equation*}
m_{(y_j),(x_i)} = f_{(x_i),(y_j)} \cdot d (\mu_{(y_j)}\times \nu_{(x_i)}) = m\circ (L_{(y_j)}\times L_{(x_i)})^{-1},
\end{equation*}
are uniquely determined by the values of the single infinite matrix
\[
r=(r_{i,j})= \big(f(x_i,y_j)\big) \in Z^{\N\times\N}.
\]

The construction of the model is done by ergodic arguments with respect to the  two--dimensional shift
on $X^\N\times Y^\N$ defined by
\[
\sigma^{(k,l)} \big((x_i), (y_j)\big) = \big((x_{i+k}), (y_{j+l}) \big),
\]
for every $k,l\geq 1$ and  $\big((x_i), (y_j)\big)$ from $X^\N\times Y^\N$, but can as well be performed with respect to the action of $\S_\N\times\S_\N$.
For every $n\geq 1$ let us choose a countable algebra $\mf A_n$ which generates the sigma algebra of all $n$--cylinders, i.e. the Borel sets formulated in the first $n$ coordinates only.
Using the ergodic theorem for the two--dimensional shift, we conclude that almost every choice of sequences $(x_i)_{i\in\N}$ and $(y_j)_{j\in\N}$ the following relations hold for every $A_1$ and $A_2$ from $\mf A=\bigcup_n \mf A_n$:
\begin{align}
\label{e:sev:mu_r}
\mu_{(y_j)} (A_1) &= \lim_{m\rightarrow\infty} \frac{1}{m}\cdot \sum_{k=1}^m \delta_{(r_{k,i})_{i=1}^\infty} (A_1),
\\
\label{e:sev:nu_r}
\nu_{(x_j)} (A_2) &=  \lim_{m\rightarrow\infty} \frac{1}{m}\cdot\sum_{l=1}^m \delta_{(r_{i,l})_{i=1}^\infty} (A_2),
\end{align}
and
\begin{multline}
\label{e:sev:m_r}
m_{(x_i),(y_j)} (A_1\times A_2)=
\\
=\lim_{m\rightarrow\infty} \frac{1}{m^2} \cdot\sum_{k,l = 1}^m
r_{k,l} \cdot \delta_{(r_{i,k})_{i=1}^\infty} (A_1) \cdot \delta_{(r_{l,j})_{j=1}^\infty} (A_2).
\end{multline}
In fact, given finitely many points $\{x_i: 1\leq i\leq n\}$ from $X$ and $\{y_j:1\leq j\leq n\}$ from $Y$,
and introducing the notation
\[
f_{(x_1,\ldots,x_n)}(y) = \big(f(x_1,y),\ldots, f(x_n,y)\big),
\]
and
\[
f_{(y_1,\ldots,y_n)}(x) = \big(f(x,y_1),\ldots, f(x,y_n)\big),
\]
then almost every continuations $(x_i)_{i=n+1}^\infty$ and $(y_j)_{j=n+1}^\infty$ of the finite sequences $(x_i)_{i=1}^n$ and $(y_j)_{j=1}^n$ are such that for any choice of sets $A_1$ and $A_2$ from $\mf A_n$,
\begin{align*}
\mu\circ f_{(y_1,\ldots y_n)}^{-1} (A_1) &= \lim_{m\rightarrow\infty} \frac{1}{m}\cdot\Big|\left\{ n+1 \leq k\leq m : f_{(y_1,\ldots,y_n)}(x_k)\in A_1\right\}\Big|,
\\
\nu\circ f_{(x_1,\ldots x_n)}^{-1} (A_2) &= \lim_{m\rightarrow\infty} \frac{1}{m}\cdot \Big|\left\{ n+1 \leq k\leq m : f_{(x_1,\ldots,x_n)}(y_k)\in A_2\right\}\Big|,
\end{align*}
and
\begin{equation*}
m\big(f_{(y_1,\ldots y_n)}^{-1}(A_1)\times f_{(x_1,\ldots,x_n)}^{-1} (A_2) \big)= \lim_{m\rightarrow\infty} \frac{1}{m^2} \cdot\sum_{(k,l)\in W\cap[n+1,m]^2} f(x_k,y_l).
\end{equation*}
Here  $W$ denotes the set
\[
W= \{(k,l)\in\N\times\N: f_{(y_1,\ldots,y_n)}(x_k)\in A_1 \text{ and }f_{(x_1,\ldots,x_n)}(y_l)\in A_2 \}.
\]
Hence the concatenated sequences $(x_i)_{i=1}^\infty$ and $(y_j)_{j=1}^\infty$ obviously satisfy the equations (\ref{e:sev:mu_r}), (\ref{e:sev:nu_r}) and (\ref{e:sev:m_r}).
Integrating over all choices of $((x_i)_{i=1}^n,(y_j)_{j=1}^n)$ from $X^n\times Y^n$ yields the sets
\begin{multline*}
\mc E_f(\mf A_n) = \Big\{\big((x_i), (y_j)\big)\in X^\N\times Y^\N: L_{(y_j)} \text{ and } L_{(x_i)} \text{ are one--to--one}\\
(\bmod\, 0) \text{ and (\ref{e:sev:mu_r}), (\ref{e:sev:nu_r}), (\ref{e:sev:m_r}) hold for all }A_1, A_2\in\mf A_n \Big\}
\end{multline*}
and hence their intersection
\begin{equation*}
\mc E_f(\mf A)=\bigcap_{n\geq 1} \mc E_f(\mf A_n)
\end{equation*}
is of full measure.

We shall call any $\big((x_i),(y_j)\big)$ from the set $\mc E_f(\mf A)$ a pair of \textit{typical sequences}.
It is well known that the image $F_f\big(\mc E_f(\mf A)\big)$ is then measurable with respect to the $D_f$--completion of the Borel algebra in the space $Z^{\N\times\N}$ thus we can find a Borel set
\[
\mc T_f(\mf A)\subseteq F_f\big(\mc E_f(\mf A)\big)
\]
which is of full measure.
We shall call this set $\mc T_f(\mf A)$ the set of \textit{typical matrices} with respect to the countable algebra $\mf A$ and the function $f$.

Summarizing these facts we have shown that for any choice of
\[
\big((x_i),(y_j)\big)
\]
from the set $\mc E_f(\mf A)$ the canonical model depends only on their matrix $(f(x_i,y_j))_{i,j}$.
Likewise, for any matrix
\[
r=(r_{i,j})_{i,j\in\N}
\]
from the set of typical matrices $\mc T_f(\mf A)$ we can construct spaces
\begin{align}
\label{e:sev:X_r and Y_r}
(X_r,\mu_r) &= \big([0,1]^\N,\mu_r\big)
\intertext{ and }
(Y_r,\nu_r) &=\big([0,1]^\N,\nu_r\big),
\end{align}
the measures $\mu_r$ and $\nu_r$ determined by the limits (\ref{e:sev:mu_r}) and (\ref{e:sev:nu_r}), and a measurable function
\begin{equation}\label{e:sev:f_r}
f_r: X_r\times Y_r\longrightarrow [0,1], \quad f_r= \frac{dm_r}{d(\mu_r\times\nu_r)},
\end{equation}
i.e. the Radon--Nikod\'ym derivative of the measure $m_r$ determined by the limit in (\ref{e:sev:m_r}) with respect to $\mu_r\times\nu_r$, which is isomorphic to the original function $f$.
Note that $\mu_r$ and $\nu_r$ are the empirical distributions of the rows and columns of the matrix $r$, and $m_r(A_1\times A_2)$ for two $n$-dimensional Borel sets $A_1$ and $A_2$ equals the average value $r_{k,l}$ which occurs when observing simultaneously the row segment $(r_{k,1},\ldots r_{k,n})$ belonging to $A_1$ and the column segment $(r_{1,l},\ldots r_{n,l})$ belonging to $A_2$.

In particular, the so constructed function $f_r$ is unique: Any matrix $r$ from the set $\mc T_f(\mf A)$ determines - up to isomorphism - the same function.
For general ergodic $(\S_\N\times \S_\N)$--invariant measures $D$ this construction might fail:
Of course the existence of the measures $\mu_r$ and $\nu_r$ follows from stationarity of the two--dimensional shift action, or of the separate actions of $S_\N\times \{1\}$ and $ \{1\}\times S_\N$ (cf. Section \ref{s:SimpleMeasures}), but it is by no means clear that the limit in (\ref{e:sev:m_r}) defines a measure $m_r$, nor that the so constructed function $f_r$ is unique up to isomorphism, contrary to what was stated in \cite{Vershik FA}.

With the above reconstruction of the canonical model we are able to show the main result of this section:

\begin{thm}[Completeness Theorem]\label{t:sev:complete invariant}
Suppose $f:X\times Y\longrightarrow Z$ and $g:X'\times Y'\longrightarrow Z$ are pure functions.
There exist measure preserving isomorphisms $S: (X,\mu)\longrightarrow (X',\mu')$ and $T: (Y,\nu)\longrightarrow (Y',\nu')$ of the corresponding standard probability spaces such that
\[
g\big(S(x),T(y)\big) = f(x,y) \quad (\bmod\, 0),
\]
if and only if their matrix distributions $D_f$ and $D_g$ coincide.
\end{thm}
\begin{proof}
One direction is trivial: If there exist such isomorphisms $S$ and $T$ then the corresponding product transformation which maps $\big((x_i)_{i},(y_j)_{j}\big)$ to $\big( (S(x_i))_{i} , (T(y_j))_{j} \big)$ is a measure preserving isomorphism of the  spaces $X^\N\times Y^\N$ and $X'^\N\times Y'^\N$ which carries the matrix valued functions $F_f$ to $F_g$. Hence their push forward measures $D_f$ and $D_g$ coincide.

Now assume that $D_f=D_g$.
As in the preceding discussion we consider $Z= [0,1]$, fix countable algebras $\mf A_n$, $n\geq 1$, which generate the Borel structure of $[0,1]^n$, and put $\mf A= \bigcup_{n} \mf A_n$.
As $D_f=D_g$ the intersection of the sets of typical matrices
\[
\mc T_f(\mf A)\cap \mc T_g(\mf A)
\]
is still of full measure and therfore it is non--empty.
But any matrix $r=(r_{i,j})$ from this intersection determines a function
\[
f_r: (X_r,\mu_r)\times (Y_r,\nu_r) \longrightarrow [0,1]
\]
which is an isomorphic model \textit{simultanously} for both functions $f$ and $g$.
This proves the existence of the claimed isomorphisms $S$ and $T$.
\end{proof}

Note that an explicit form of the isomorphisms from Theorem \ref{t:sev:complete invariant} is
\[
S = L_{(y_j')}^{-1}\circ L_{(y_j)} \text{ and } T = L_{(x_i')}^{-1}\circ L_{(x_i)},
\]
the pairs $\big((x_i),(y_j)\big)$ and $\big((x_i'),(y_j')\big)$ being any two pairs of typical sequences, i.e.  from
\[
\mc E_f(\mf A)\cap \mc E_g(\mf A),
\]
which define the same matrix: $\big(f(x_i,y_j)\big)= \big(g(x_i',y_j')\big)$.
This observation will be useful in the proof of the following corollary, which will be needed
in Section \ref{s:SimpleMeasures}.

\begin{cor}\label{c:sev:aut}
Let $f:X\times Y\longrightarrow Z$ be pure. Then the map $F_f:X^\N\times Y^\N\longrightarrow Z$ is one--to--one ($\bmod\, 0$) if and only if its congruence group $K_f$ is trivial.
\end{cor}
\begin{proof}
As before we assume that $Z=[0,1]$.
Assume that $F_f$ is not one--to--one ($\bmod\, 0$).
If $\mf A$ is any arbitrary countable algebra generating the Borel structure of $[0,1]^\N$ then the map $F_f$ cannot be injective on $\mc E_f(\mf A)$, since this set is of full measure.
Thus there exist two different pairs $\big((x_i),(y_j)\big)$ and $\big((x_i'),(y_j')\big)$ from $\mc E(\mf A)$ which have the same image under $F_f$.
Therefore the mappings
\[
S = L_{(y_j')}^{-1}\circ L_{(y_j)} \text{ and } T = L_{(x_i')}^{-1}\circ L_{(x_i)},
\]
are measure preserving automorphisms of $(X,\mu)$ and $(Y,\nu)$ respectively for which $f(Sx,Ty)=f(x,y)$ ($\bmod\, 0$). Moreover these automorphisms satisfy $f(Sx , y_j)= f(x,y_j')$ for a.e. $x$ in $X$ and $f(x_i,Ty)= f(x_i',y)$ for a.e. $y$ in $Y$, $i$ and $j$ being arbitrary.
This proves that either $S$ or $T$ must be non-trivial since we may assume that both mappings $f_X$ and $f_Y$ (defined in Section \ref{s:BasicDefinitions}) are one--to--one.

The other direction is trivial: Assume that there exists a non--trivial automorphism $(S,T)$ in $K_f$. Then every set $B\subseteq X^\N\times Y^\N$ of full measure is almost invariant under $(S,T)^\N$. But at the same time the automorphism $(S,T)^\N$ leaves $F_f$ almost invariant and therefore we can find two different points in $B$ with the same image under $F_f$.
\end{proof}

\section{Absence of symmetry - simple measures}
\label{s:SimpleMeasures}

In this section we characterize those matrix distributions which originate from a function with no symmetries, i.e. with trivial congruence group.
Its main result, Theorem \ref{t:simple measures} corrects the corresponding statement Theorem 3 in \cite{Vershik FA}.

\begin{defn}
Let $D$ be a $\left(\S_\N\times\S_\N\right)$--invariant and ergodic measure on the space of matrices.
We say that $D$ is \textit{simple} if the invariant algebras $\mf B^{S_\N\times \{1\}}$ and $\mf B^{ \{1\}\times S_\N}$ generate ($\bmod\, 0$) the whole Borel algebra $\mf B$ of $Z^{\N\times\N}$, i.e.
\begin{equation*}\label{e:sev:simple}
\mf B^{S_\N\times \{1\}}\vee \mf B^{ \{1\}\times S_\N} = \mf B \quad (\bmod\, D).
\end{equation*}
\end{defn}

In other words simplicity of an $\left(\S_\N\times\S_\N\right)$--invariant and ergodic measure $D$ means that the dynamical system
\[
\big(Z^{\N\times\N},D,\S_\N\times\S_\N\big)
\]
 decomposes into the direct product
\begin{equation}
\label{e:decomposition}
\big(Z^{\N\times\N},D,\S_\N\times\S_\N\big) = \left(Z^{\N\times\N}/G_2, D, G_1\right) \times \left(Z^{\N\times\N}/G_1, D, G_2\right)
\end{equation}
of the systems of the ergodic components with respect to the commuting subgroups
\[
G_1= S_\N\times \{1\} \quad\text{and}\quad G_2= \{1\}\times S_\N,
\]
which act seperately on the rows and columns of the matrices.
In fact, we choose the standard Borel models
\begin{align}
\label{e:ComponentXX}
\left(Z^{\N\times\N}/G_2, D\right) &= \big(\meas_1\left(Z^{\N\times \N}\right), D\circ \pi_1^{-1}\big),
\\
\label{e:ComponentYY}
\left(Z^{\N\times\N}/G_1, D\right) &= \big(\meas_1\left(Z^{\N\times\N}\right), D\circ \pi_2^{-1}\big),
\end{align}
for the measure space of ergodic components, together with the isomorphism
\begin{equation}
\label{e:decompositionIsomorphisms}
\begin{aligned}
\pi: Z^{\N\times\N} &\longrightarrow Z^{\N\times\N}/G_2 \times Z^{\N\times\N}/G_1,
\\
r &\mapsto \left(\pi_1(r),\pi_2(r)\right) = \left(D_r^{G_2}, D_r^{G_1}\right),
\end{aligned}
\end{equation}
given by the decomposition of $D$ into its $G_2$--ergodic and $G_1$--ergodic measures, respectively.
As we may assume without loss in generality the mappings $r\mapsto D_r^{G_2}$ and $r\mapsto D_r^{G_1}$ invariant with respect to $G_2$ and $G_1$ respectively, the equations
\begin{align*}
g_1\cdot D_r^{G_2}  &= D_{g_1\cdot r}^{G_2},
\\
g_2 \cdot D_r^{G_1}  &= D_{g_1\cdot r}^{G_1},
\end{align*}
with $g_1\in G_1$ and $g_2\in G_2$,
define unambigously actions of the permutation groups $G_1$ and $G_2$ on the respective spaces of ergodic components so that the product of these actions is a factor of the action of $\S_\N\times\S_\N$ on $Z^{\N\times\N}$.

Note that the points $\mf x$ and $\mf y$ of the factor spaces $\mf X$ and $\mf Y$ are represented by permutation--invariant and ergodic measures $D_r^{G_2}$ and $D_r^{G_1}$ which are therfore Bernoulli measures by de Finetti's theorem.
Concretely, when considering the space $Z^{\N\times\N}$ as the `space of sequences of rows' $(Z^{\{1\}\times\N})^\N$  or  the `space of sequences of columns' $(Z^{\N\times\{1\}})^\N$  then
\begin{equation*}
D_r^{G_2}=\nu_r^\N  \quad\text{and}\quad D_r^{G_1} = \mu_r^\N
\end{equation*}
the measures $\nu_r$ and $\mu_r$  being a probability measure the space of columns and rows, respectively.
We henceforth will not distinguish the spaces of ergodic components in \eqref{e:decomposition} from the spaces
\begin{align}
\label{e:ComponentX}
(\mf X,\mu_{\mf X},G_1) &= \left(\meas_1\left(Z^{\N\times\{1\}}\right), D\circ \pi_1^{-1},G_1\right),
\\
\label{e:ComponentY}
(\mf Y,\nu_{\mf Y},G_2) &= \left(\meas_1\left(Z^{\{1\}\times\N}\right), D\circ \pi_2^{-1},G_2\right),
\end{align}
the isomorphism $\pi=(\pi_1,\pi_2)$ interpreted as mapping
\begin{equation}
\label{e:IsomorphismDecomposition}
\pi: Z^{\N\times\N}\longrightarrow \mf X\times \mf Y, \quad r\mapsto (\nu_r,\mu_r),
\end{equation}
whenever appropriate.
Notice that the measures $\nu_r$ and $\mu_r$ describe $(\bmod~0)$ the empirical distribution of  columns and rows of the matrix $r$.
The group actions on $\mf X$ and $\mf Y$ take their explicit form as
\begin{align*}
g_1 (\mu_r) &= \nu_{g_1(r)},
\\
g_2 (\nu_r) &=\mu_{g_2(r)},
\end{align*}
for any $g_1$ in $G_1$ and  $g_2$ in $G_2$.
Further note that by invariance and ergodicity of the measure $D$ with respect to $G_1\times G_2$, the measures $\mu_\mf X$ and $\nu_\mf Y$ are invariant and ergodic with respect to the corresponding actions of $G_1$ and $G_2$.

If the measure $D$ is not simple then the map $\pi$ is not one--to--one on a set of full measure.
In this case the direct product of $(\mf X,\mu_{\mf X},G_1)$ and $(\mf Y,\nu_{\mf Y},G_2)$ is a non--trivial factor of $\big(Z^{\N\times\N},D,G_1\times G_2\big)$.
This gives us the following characterisation of simple measures:
\begin{prop}
A measure $D$ is simple if and only if the (almost everywhere uniquely defined) mapping which sends a matrix $r$ to the pair $(\nu_r,\mu_r)$ of its empirical distribution of columns and rows respectively, is one--to--one ($\bmod\, 0$).
\end{prop}

With help of the decomposition of the group action we are able to proof the following lemma.

\begin{lem}\label{l:sev:simple}
Assume that $D=D_f$ is the matrix distribution of a pure function $f:(X,\mu)\times (Y,\nu)\longrightarrow Z$.
Then $D$ is simple if and only if the congruence group $K_f$ is trivial.
\end{lem}
\begin{proof}
First of all note that
by Corollary \ref{c:sev:aut}, we only have to show that a measure $D$ is simple if and only if the map
\[
F_f: (X^\N,\mu^\N)\times (Y^\N,\nu^\N)\longrightarrow(Z^{\N\times\N},D)
\]
is one--to--one  on a set of full measure.
One direction is trivial:
if $F_f$ is one--to--one, it is an equivariant isomorphism between the measure spaces.
Hence
\begin{equation}
\label{e:simplemeasure}
\mf B^{S_\N\times \{1\}}\vee \mf B^{ \{1\}\times S_\N} = \mf B \quad (\bmod\, D),
\end{equation}
as the same assertion is true for the corresponding invariant algebras in the space $X^\N\times Y^\N$.

Conversely, assume that \eqref{e:simplemeasure} holds.
By the equality
\[
\big(Z^{\N\times\N},D,G_1\times G_2\big) =(\mf X,\mu_{\mf X},G_1)\times(\mf Y,\nu_{\mf Y},G_2)
\]
we thus may regard $F_f$ as mapping
\[
(X^\N,\mu^\N, G_1) \times (Y^\N,\nu^\N, G_2) \longrightarrow (\mf X,\mu_{\mf X},G_1)\times (\mf Y,\nu_{\mf Y},G_2).
\]
By equivariance, the preimage $F_f^{-1}(B)$ of any $G_i$--invariant Borel set $B\subseteq \mf X\times\mf Y$ is also invariant, for any $i=1,2$.
Thus $F_f$ must be of the form
\[
F_f = (\phi_1,\phi_2),
\]
with $\phi_1: X^\N\longrightarrow \mf X$ and $\phi_1: Y^\N\longrightarrow \mf Y$.
By Lemma \ref{l:sev:pure} we know that for $\nu^\N$--almost every sequence $(y_j)$ the restriction
$F_f(\:.\:,(y_j))$
is one--to--one ($\bmod\, 0$) and so is $\phi_1$.
For the same reasoning the function $\phi_2$ is one--to--one ($\bmod\, 0$).
This proves that $F_f=(\phi_1,\phi_2)$ is one--to--one on a set of full measure.
\end{proof}

\subsection{General canonical model given the matrix distribution of a completely pure function}
Assume that $D$ is the matrix distribution of a \emph{completely pure function} $f$, that is a pure function $f$ with trivial congruence group $K_f$.
Then by Lemma \ref{l:sev:simple} the measure $D$ is simple and hence the decomposition
\[
(Z^{\N\times\N}, D,\S_\N\times\S_\N) = \big(\mf X,\mu_{\mf X},G_1\big) \times \big(\mf Y,\nu_{\mf Y}, G_2 \times\S_\N\big)
\]
described by \eqref{e:decomposition}, \eqref{e:ComponentX}, \eqref{e:ComponentY} together with its
isomorphism
\[
\pi: Z^{\N\times\N}\longrightarrow \mf X\times \mf Y
\]
from \eqref{e:IsomorphismDecomposition} gives us another possibility to reconstruct the function $f$:
We simply set
\begin{equation*}
\mf f:\mf X\times \mf Y\longrightarrow Z, \quad \mf f =  r_{1,1}\circ \pi^{-1},
\end{equation*}
and claim that it is isomorphic to the function
\[
f:X^\N\times Y^\N\longrightarrow Z
\]
regarded as function on the first coordinates $x_1$ and $y_1$.
In fact, the mapping
\[
\Phi=\pi\circ F_f: X^\N\times Y^\N\longrightarrow \mf X\times \mf Y
\]
is one--to--one, measure preserving and obviously carries the function $f:X^\N\times Y^\N\longrightarrow Z$ to $\mf f$.
Thus the only thing we need to check to is that $\Phi$ is of product type, i.e. $\Phi=(\Phi_1,\Phi_2)$ with measurable functions $\Phi_1: X^\N\longrightarrow \mf X$ and $\Phi_2: Y^\N\longrightarrow\mf Y$.
But this is clear from the proof of Lemma \ref{l:sev:simple}.

We shall call the above constructed function $\mf f$ the \textit{general canonical model} of $f$, since it does not depend on the choice of a particular matrix $r$ as the individual canonical model from Section \ref{s:complete invariant}.
In contrast to the individual canonical model the function $\mf f$ is never pure as it is a model for $f$ as function on $X^\N\times Y^\N$ rather than as function $X\times Y\longrightarrow Z$.
Nevertheless its purification
\[
\bar{\mf f}: \mf X/\zeta_{\mf X} \times \mf Y/\zeta_{\mf Y}\longrightarrow Z
\]
as described in Section \ref{s:PureFunctions} is clearly isomorphic to the original pure function
\[
f:X\times Y\longrightarrow Z
\]
by the uniqueness of pure factors.
The direct connection between both models becomes clear in the proof of Theorem \ref{t:simple measures}.


Of course the above construction of  $\mf f$ makes sense for any $(\S_\N\times\S_\N)$--invariant simple measure on $Z^{\N\times\N}$.
This observation is the key for the following theorem.

\begin{thm}[Characterisation of simple measures using the general canonical model $\mf f$]
\label{t:simple measures}
Let $D$ be an $(\S_\N\times\S_\N)$--invariant and ergodic probability measure on the space of matrices $Z^{\N\times\N}$.
Then $D$ is a matrix distribution of a function $f:(X,\mu)\times (Y,\nu)\longrightarrow Z$ with trivial congruence group $K_f$ if and only if $D$ is a simple measure.
\end{thm}
\begin{rem}
This theorem corrects Theorem 2 from \cite{Vershik FA}, which states that an $(S_\N\times S_\N)$--invariant and ergodic measure is a matrix distribution if and only if it is simple.
\end{rem}

\begin{rem}
Observe that the characterization of matrix distributions in the case of functions in just one argument becomes essentially de Finetti's theorem:
In this context, the \textit{row distribution} (instead of matrix distribution) $D_f$ of $f:X\longrightarrow Z$ is defined as the distribution of the process $(f(x_i))$ sampling its arguments independently according to the given measure $\mu$ on $X$.
Thus $D_f$ is simply the Bernoulli measure
$D=\left(\mu\circ f^{-1}\right)^\N$
on the space $Z^\N$.
De Finetti's theorem states that the $\S_\N$--invariant ergodic measure on $Z^\N$ are exactly the Bernoulli measures, and hence all such meausres are the row distribution of a function in one variable.
\end{rem}

\begin{proof}
One direction is already content of Lemma \ref{l:sev:simple}.
Conversely, let us assume that (\ref{e:sev:simple}) is satisfied.
As before the dynamical system $(\Z^{\N\times\N},D,\S_\N\times\S_\N)$ decomposes into the product
\[
(\mf X, \mu_{\mf X},G_1) \times (\mf Y,\nu_\mf Y,G_2)
\]
of the spaces of ergodic components
\begin{align*}
\left(\mf X,\mu_\mf X\right) &= \left(\meas_1\left(Z^{\N\times\{1\}}\right), D\circ \pi_1^{-1}\right),
\\
\left(\mf Y,\nu_\mf Y\right)  &= \left(\meas_1\left(Z^{\{1\}\times\N}\right),D\circ \pi_2^{-1}\right),
\end{align*}
modelled by the space of column and row measures, respectively.
Taking a closer look on these spaces, it is obvious that the ergodic decomposition mapping $\pi_1: r\mapsto D_r^{G_2}$ regarded as $G_1$--equivariant mapping
\[
\pi_1: \left(\Z^{\{1\}\times\N}\right)^\N\longrightarrow \mf X
\]
maps each of the $G_1$--invariant ergodic measures $D_r^{G_1}$ onto an $G_1$--invariant ergodic measure on $\mf X$.
As $\mu_\mf X$ is also an $G_1$--invariant ergodic measure, being the integral convex combination
\[
\mu_\mf X(B) = \int_{r} D_r^{G_1}\circ \pi_1^{-1}(B) \cdot dD(r)
\]
of such measures, it  therefore must equal to at least\footnote{it is not difficult to see that the set of all mesures $D_r^{G_1}$ which equal the ergodic measure $\mf X$ is of full measure.} one such pushforward measures $D_r^{G_1}\circ\pi_1^{-1}$.
By de Finetti's theorem we may regard the latter measure $D_r^{G_1}$ as Bernoulli measure $\mu_r^\N$ on the space of sequences of rows $\left(\Z^{\{1\}\times\N}\right)^\N$, and $\pi_1$ as measure preserving isomorphism between the measure spaces $\left(\Z^{\{1\}\times\N}\right)^\N$ and $\mf X$.
We thus may consider
\[
\left(\mf X,\mu_\mf X\right)=\left(\left(\Z^{\{1\}\times\N}\right)^\N,\mu_r^\N\right),
\]
and in the same way one sees that
\[
\left(\mf Y,\mu_\mf Y\right)=\left(\left(\Z^{\N\times\{1\}}\right)^\N,\nu_r^\N\right),
\]
that is the space of sequences of columns with Bernoulli measure $\nu_r^\N$.
We claim that the matrix distribution of the function
\[
\mf f:  \mf X\times \mf Y \longrightarrow Z, \quad \mf f= r_{1,1}\circ \pi^{-1},
\]
equals our measure $D$.
For brevity, we write  $X=Z^{\{1\}\times\N}$ and $Y=Z^{\N\times\{1\}}$ for the space of rows and the space of columns, respectively.
As the matrix valued function
\[
\mf F:\mf X\times \mf Y\longrightarrow Z^{\N\times\N},\quad \mf F =  \id\circ\pi^{-1},
\]
is equivariant with respect to the action of $G_1\times G_2$ so it is regarding it as function from $X^\N\times Y^\N$ to $Z^{\N\times\N}$.
From this it follows easily that $\mf F$ is of the form
\[
\mf F\big((x_i),(y_j)\big) = \big( \mf F_{1,1}(x_i,x_j)\big) \quad (\bmod\,  0),
\]
as is shown in the postponed Lemma \ref{l:sev:equivariance}.
This proves that the distribution of the matrix valued function $\mf F$, which by definition equals $D$, is the matrix distribution of  $\mf f$.
Moreover, the function $\mf F=F_{\mf f}$ is by construction one--to--one  ($\bmod\, 0$) and we conclude from Corollary \ref{c:sev:aut} that the congruence group of $\mf f$ is trivial.
\end{proof}

We close this section with the auxiliary lemma that we used in the proof of Theorem \ref{t:simple measures}.
\begin{lem}\label{l:sev:equivariance}
Suppose that the map $F$ from $(X^\N,\mu^\N)\times (Y^\N,\nu^\N)$ to $Z^{\N\times\N}$ is equivariant under the action of $S_\N\times S_\N$.
Then there exists a measurable function $f:X\times Y\longrightarrow Z$ such that
\begin{equation*}\label{e:sev:L07}
F=(F_{i,j}) = \big(f(x_i,x_j)\big) \quad (\bmod\,  0).
\end{equation*}
\end{lem}
\begin{proof}
Using equivariance we know that for any $g_1$ and $g_2$ from the subgroup $S_\N^{(1)}=\{g\in\S_\N:g(1)=1\}$ the following identity holds.
\[
\mf F_{1,1}\big(g_1(x_i),g_2(y_j)\big) = \mf F_{g_1^{-1} (1),g_2^{-1} (1)}\big((x_i),(y_j)\big) =\mf F_{1,1}\big((x_i),(y_j)\big).
\]
Being invariant with respect to the action of $S_\N^{(1)}\times S_\N^{(1)}$ the function $\mf F_{1,1}$ is ($\bmod\, 0$) measurable with respect to the first coordinates $x_1$ and $y_1$ of $X_r^\N\times Y_r^\N$ which means that $\mf F_{1,1}\big((x_i),(y_j)\big) = f(x_1,y_1)$ ($\bmod\, 0$) for some function $f$ defined on $X\times Y$.
Using once more equivariance we conclude that
\begin{multline*}
F_{g_1^{-1} (1),g_2^{-1} (1)}\big((x_i),(y_j)\big) =\\
=
F_{1,1}\left(\left(x_{g_1^{-1}(i)}\right),\left(y_{g_2^{-1}(j)}\right)\right) = f\left(x_{g_1^{-1}(1)},y_{g_2^{-1}(1)}\right)
\end{multline*}
($\bmod\, 0$) for every $g_1$ and $g_2$ from $S_\N$, which proves the assertion of the lemma.
\end{proof}

\section{The case of functions in more than two arguments}\label{s:more than two}

Let us shortly discuss the case of Borel functions
\[
f:\prod_{i=1}^n (X_i,\mu_i) \longrightarrow Z
\]
in more than two arguments.
Its \textit{tensor distribution} $D_f$ is defined as the distribution of the tensor--valued functional
\begin{align*}
F_f : \prod_{i=1}^n X_i^\N  &\longrightarrow Z^{\N^n},
\\
\left(\left(x_{i}^{(1)}\right), \left(x_{i}^{(2)}\right),\ldots, \left(x_{i}^{(n)}\right)\right) &\mapsto \left(f\left(x_{i_1}^{(1)},\ldots, x_{i_n}^{(n)}\right)\right)_{i_1,\ldots, i_n},
\end{align*}
under the measure $\prod_{i=1}^n \mu_i^\N$, i.e.
$D_f =\left(\prod_{i=1}^n \mu_i^\N\right)\circ F_f^{-1}$.

In this situation $D_f$ is a measure on the space of infinite tensors $Z^{\N^n}$ which is ergodic and invariant with respect to the (analogously defined action) of
\[
\S_\N^n = \S_\N\times\S_\N\times\ldots\times\S_\N
\]
acting independently on the indices of the tensors.
The notion of factors in the category of measurable functions in several variables is analogous to that of
functions in two arguments, and so is the definition of pureness:
\textit{%
a function $f(x_1,x_2,\ldots,x_n)$ is said to be pure, if it admits no true factor.
}
The reduction of the isomorphism problem to pure functions is then proved in exactly the same way as Theorem \ref{t:ReductionToPureCase}.
With help of an extended version of Lemma \ref{l:pure} it is also obvious how to prove the higher--dimensional analogue of Theorem \ref{t:sev:complete invariant}: \textit{Two pure measurable functions are isomorphic if and only if their tensor distributions coincide.}

The results from Section \ref{s:SimpleMeasures} are also extended easily:
Let $G_i$ be the permutation group which acts on the $i$-th index of the tensors only.
An $\left(S_\N^n\right)$--invariant measure $D$ on the space of tensors $Z^{\N^n}$ is said to be simple, if the invariant algebras $\mf B^{G^{(i)}}$ with respect to the groups
\[
G^{(i)} = G_1\times \cdots \times G_{i-1} \times \{1\} \times G_{i+1}\times \cdots \times G_n
\]
keeping the $i-th$ index fixed, generate the whole Borel algebra $\mf B$ of $Z^{\N^n}$, i.e.
\[
\bigvee_{i=1}^n \mf B^{G^{(i)}} = \mf B \quad (\bmod\, D).
\]
This again means that the dynamical system $(Z^{\N^n}, D, \prod_{i=1}^n G_i)$ is the direct product of the systems
$(Z^{\N^n}/ G^{(i)}, D, G_i)$, $1\leq i\leq n$, the quotients interpreted as ergodic decompositions.
With help of this decomposition, the general canonical model is defined analogously and  Theorem \ref{t:simple measures} can be proved similary:
\textit{An $S_\N^n$-invariant and ergodic measure on the space of tensors $Z^{\N^n}$ is the tensor distribution of measurable function $f:\prod_{i=1}^n (X_i,\mu_i)\longrightarrow Z$ with trivial congruence group
\begin{multline*}
K_f= \Big\{(T_i)_{i=1}^n \in \prod_{i=1}^n \Aut_0(X_i,\mu_i):
\\
f\big(T_1(x_1) ,\ldots, T_n(x_n)\big)=f(x_1,\ldots, x_n) \quad \text{a.e.} \Big\}
\end{multline*}
if and only it is simple.
}

\end{document}